\documentclass[12pt]{amsart}
\usepackage{setspace}
\usepackage{amsmath, amssymb,amsthm, amsfonts, latexsym, mathrsfs}
\usepackage{hyperref}
\usepackage{color}
\usepackage{graphicx}

\newtheorem{thm}{Theorem}[section]
\newtheorem*{thm*}{Theorem}
\newtheorem{lemma}{Lemma}[section]
\newtheorem{prop}{Proposition}[section]

\newtheorem{df}{Definition}[section]
\newtheorem{ques}{Question}[section]

\theoremstyle{remark}
\newtheorem{rem}{Remark}[section]

\newcommand{\Ric}{\mbox{Ric}}
\newcommand{\R}{\mathbb R}
\newcommand{\Z}{\mathbb Z}

\numberwithin{equation}{section}
\newcommand{\be}{\begin{equation}}
\newcommand{\ee}{\end{equation}}
\def\p{\partial}
\def\M{\mathcal{M}}

\def\lf{\left}
\def\ri{\right}
\def\Pi{\displaystyle{\mathbb{II}}}
\def\Ric{\text{\rm Ric}}

\def\m{\mathfrak{m}}

\def\bee{\begin{equation*}}
\def\eee{\end{equation*}}

\def\mby{\mathfrak{m}_{_{BY}}}

\def\H{\mathbb{H}}

\def\F{\mathring{\mathcal{F}}}
\def\L{\mathring{\Lambda}}

\def\So{\Sigma_{O}}
\def\Sh{\Sigma_{H}}

\def\FF{\mathcal{F}}
\def\LL{\Lambda}

\newcounter{mnotecount}[section]

\begin{document}

\title{
Total mean curvature, scalar curvature, and a variational analog of Brown-York mass
}

\author{Christos Mantoulidis}
\address[Christos Mantoulidis]{Department of Mathematics, Stanford University, Stanford, CA 94305,   USA.}
\email{c.mantoulidis@math.stanford.edu}

\author{Pengzi Miao}
\address[Pengzi Miao]{Department of Mathematics, University of Miami, Coral Gables, FL 33146, USA.}
\email{pengzim@math.miami.edu}

\thanks{The first named author's research was partially supported by the Ric Weiland Graduate Fellowship at Stanford University. The second named author's research was partially supported by the Simons Foundation Collaboration Grant for Mathematicians \#281105.}


\date{}

\begin{abstract}
We study  the supremum of the total mean curvature on the boundary of compact, mean-convex 3-manifolds with nonnegative scalar curvature,  and a  prescribed  boundary metric. 
We establish an additivity property for this supremum  and   exhibit rigidity for maximizers assuming the supremum is attained.
 When the boundary consists of 2-spheres, we demonstrate that the finiteness of the supremum 
 follows from the previous work of Shi-Tam and Wang-Yau on the quasi-local mass problem in general relativity. 
In turn, we define  a variational analog of Brown-York quasi-local mass 
without assuming that  the boundary 2-sphere has positive Gauss curvature.  
\end{abstract}

\maketitle

\markboth{Christos Mantoulidis and Pengzi Miao}
{Total mean curvature, scalar curvature}

\section{Introduction and statement of results}

Brown and York (\cite{BY1}, \cite{BY2}) formulated a definition of quasi-local mass in general relativity
by employing a Hamilton-Jacobi analysis of the Einstein-Hilbert action. 
Given  a compact spacelike hypersurface $\Omega$ in a spacetime, assuming its boundary $\p \Omega$ is a 2-sphere with positive Gauss curvature,
 the Brown-York mass of $\p \Omega$ is given by 
\be \label{eq-BY}
 \m_{_{BY}}(\p \Omega; \Omega) = \frac{1}{8\pi} \int_{\partial \Omega} (H_0 - H) \, d\sigma .
\ee
Here $d\sigma$ is the induced area element on $\p \Omega$, $H$ is the mean curvature of $\p \Omega$ in $\Omega$, 
and $H_0$ is the mean curvature of the isometric embedding of $\p \Omega$ into Euclidean space, $ \R^3$. 
The existence and uniqueness of such an embedding of $ \p \Omega$ is guaranteed when $\p \Omega $ has positive  Gauss curvature, by the solution to Weyl's embedding problem (\cite{Nirenberg53, Pogorelov}).

 In \cite{ShiTam02},  Shi and Tam  proved  the following theorem
 which implies the positivity of $\m_{_{BY}}(\p \Omega; \Omega) $ when  $\Omega$ 
has nonnegative scalar curvature.

\begin{thm}[\cite{ShiTam02}] \label{thm-ShiTam} 
Let $(\Omega, g)$ be a compact, connected, Riemannian 3-manifold with nonnegative scalar curvature, and with nonempty boundary $\p \Omega$.
Suppose $\p \Omega$ has finitely many  components $ \Sigma_j$, $ j = 1, \ldots, k $, so that each $ \Sigma_j$  is a topological 2-sphere which has positive Gauss curvature and positive mean curvature $H_{g,j}$.
Then 
\be \label{eq-ST} 
\int_{\Sigma_j} H_{g,j} d \sigma_j \le \int_{{\Sigma}_j} {H}_{0,j} d {\sigma_j},
\ee
where $ d \sigma_j $ denotes the induced area element on $\Sigma_j$, and $H_{0,j}$ is the mean curvature of the isometric embedding of $\Sigma_j$ in $\R^3$.
Moreover, equality in \eqref{eq-ST} holds for some  $\Sigma_j$ if and only if $\partial \Omega$ has a unique connected component and $(\Omega, g)$ is isometric to a convex domain in $\R^3$.
\end{thm}

\begin{rem}
Our convention for the mean curvature $H$  is that mean convexity is equivalent to $H > 0$. We will often emphasize the dependence of a mean curvature on the interior metric by using the metric as a subscript; e.g., $H_g$.
\end{rem}

In this paper, we consider questions concerning  the total boundary mean curvature  of a general compact  Riemannian 3-manifold 
with nonnegative scalar curvature which are motivated by Theorem \ref{thm-ShiTam}. 
As an application of our results, we give a variational analog of Brown-York mass that is free of the positive Gauss curvature restriction on the boundary.

\subsection{Variational results}

We first introduce the relevant definitions  before stating the main results. 

\begin{df}[Fill-ins]
\label{df-F-multiple}
Let $\Sigma_1, \ldots, \Sigma_k$ be $k \geq 1$ closed, connected, orientable surfaces endowed with Riemannian metrics $\gamma_1, \ldots, \gamma_k$. 
Denote by $\FF_{(\Sigma_1, \gamma_1), \ldots, (\Sigma_k, \gamma_k)}$ the set of all compact, connected Riemannian 3-manifolds $(\Omega, g)$ with boundary such that:
\begin{enumerate}
\item $\partial \Omega$, with the induced metric, is isometric to the disjoint union of $(\Sigma_j, \gamma_j)$, $j = 1, \ldots, k$,
\item $\partial \Omega$ is mean-convex; i.e., the mean curvature vector of $\partial \Omega$ points inward, and 
\item $R(g) \geq 0$, where $R(g)$ is the scalar curvature of $g$.
\end{enumerate}
\end{df}

\begin{rem}
In Definition \ref{df-F-multiple}, we do not prescribe the mean curvature on the boundary, other than to require it to be positive. 
For a definition of ``fill-ins" that prescribes $H$, we refer the reader to  \cite{JJ}.
\end{rem}

Given the set $\FF = \FF_{(\Sigma_1, \gamma_1), \ldots, (\Sigma_k, \gamma_k)}$, we define
\be \label{eq-def-L-multiple}
\LL_{(\Sigma_1, \gamma_1), \ldots, (\Sigma_k,\gamma_k)} := \sup \lf\{ \frac{1}{8 \pi}  \int_{\p \Omega} H_g d \sigma \ | \ (\Omega, g) \in \FF \ri\} .
\ee
\begin{rem}
    The supremum of an empty set is conventionally $-\infty$.
\end{rem}
In this notation, Theorem \ref{thm-ShiTam} implies that when $\Sigma$ is a 2-sphere and $\gamma_1, \ldots, \gamma_k$ are Riemannian metrics on $\Sigma$ with positive Gauss curvature, then for every fill-in $(\Omega, g) \in \FF_{(\Sigma_1,\gamma_1), \ldots, (\Sigma_k, \gamma_k)}$, 
\be \label{eq-ST-notation} \int_{\Sigma_j} H_{g,j} \, d\sigma_j \leq 8 \pi \LL_{(\Sigma, \gamma_j)} = \int_{\Sigma_j} H_{0,j} d\sigma_j
\ee
for all $j = 1, \ldots, k$ and, therefore,
\[ \LL_{(\Sigma,\gamma_1),\ldots,(\Sigma,\gamma_k)} \leq \sum_{j=1}^k \LL_{(\Sigma,\gamma_j)} \text{.} \]
Moreover, if equality holds in \ref{eq-ST-notation} for some $j$, then $k = 1$ and $(\Omega, g)$ is necessarily the unique flat fill-in $(B^3, \delta) \in \FF_{(\Sigma, \gamma_1)}$ from Weyl's embedding problem.

The geometric significance of inequality \eqref{eq-ST-notation} is in that its right side is a quantity  that is determined only by the induced metric on the boundary component $\Sigma_j \subset \p \Omega$; it is independent of the interior of the manifold, and independent of all other boundary components.

\begin{ques}
If we drop the positivity requirement on the Gauss curvature of the boundary 2-spheres, is there still a bound for the total mean curvature of the boundary that is intrinsic? Can we characterize maximizers of the total mean curvature?
\end{ques}

Regarding the functional  $\LL_{(\Sigma_1, \gamma_1), \ldots, (\Sigma_k, \gamma_k)}$, we have

\begin{thm}[Additivity]
\label{thm-main-additivity}
    Let $\Sigma_1, \ldots, \Sigma_k$ be $k \geq 2$ closed, connected, orientable surfaces endowed with Riemannian metrics $\gamma_1, \ldots, \gamma_k$. For all fill-ins $(\Omega, g) \in \FF_{(\Sigma_1, \gamma_1), \ldots, (\Sigma_k, \gamma_k)}$, and for every $j = 1, \ldots, k$,
\be \label{ineq-single-sup}
 \int_{\Sigma_j} H_{g,j} \, d\sigma_j \leq 8 \pi \LL_{(\Sigma_j, \gamma_j)} \text{,} 
 \ee
where $d \sigma_j$ denotes the area element on $\Sigma_j$. Moreover,
\be \label{eq-additivity} \LL_{(\Sigma_1, \gamma_1), \ldots, (\Sigma_k, \gamma_k)} = \sum_{j=1}^k \LL_{(\Sigma_j,\gamma_j)} \text{,}
\ee
provided $\FF_{(\Sigma_j, \gamma_j)}$, $j = 1, \ldots, k$, and $\FF_{(\Sigma_1,\gamma_1),\ldots,(\Sigma_k,\gamma_k)}$ are all nonempty.
\end{thm}

\begin{rem}
In the course of the proof we will see that $\FF_{(\Sigma_1,\gamma_1),\ldots,(\Sigma_k,\gamma_k)}$ is empty if and only if some $\FF_{(\Sigma_j,\gamma_j)}$, $j = 1, \ldots, k$, is empty.  
\end{rem}

The theorem above roughly allows us to reduce our study to $k = 1$ boundary component. 
We prove this theorem by employing a cut-and-fill technique for manifolds with positive scalar curvature.

\begin{thm} [Finiteness] \label{thm-main-finiteness}
If $\Sigma$ is a 2-sphere endowed with an arbitrary Riemannian metric $\gamma$, then 
$$\LL_{(\Sigma, \gamma)} < \infty.$$
Consequently, for all Riemannian metrics $\gamma_1, \ldots, \gamma_k$ on a 2-sphere $\Sigma$, 
\[ \LL_{(\Sigma,\gamma_1),\ldots,(\Sigma,\gamma_k)} < \infty \text{.} \]    
\end{thm}

This theorem is an important ingredient of our variational analog of the Brown-York mass. We prove it by making use of results of Wang-Yau \cite{WangYau06} and Shi-Tam \cite{ShiTam06}.

\begin{rem}
Let $\Sigma$ be a 2-sphere. It is easily seen that $\FF_{(\Sigma, \gamma)} \neq \emptyset$
 if the metric $\gamma$ has positive Gauss curvature. It is also the case that $\FF_{(\Sigma,\gamma)} \neq \emptyset$ for certain metrics $\gamma$ with arbitrarily negative portions (in the $L^1$ sense) of curvature; indeed, by employing the method in  \cite{M-S}, one can show that 
  $\FF_{(\Sigma, \gamma)} \neq \emptyset$ for every metric $\gamma$ on $\Sigma$ with $\lambda_1(-\Delta_\gamma + K_\gamma) > 0$, where $\Delta_\gamma$ and $K_\gamma$ are the Laplace-Beltrami operator and the Gauss curvature of $(\Sigma, \gamma)$, respectively.
\end{rem}

\begin{thm}[Rigidity]
\label{thm-main-rigidity}
Let $\Sigma_1, \ldots, \Sigma_k$ be $k \geq 1$ closed, connected, orientable surfaces endowed with Riemannian metrics $\gamma_1, \ldots, \gamma_k$. 
If there exists $(\Omega, g) \in \FF_{(\Sigma_1, \gamma_1), \ldots, (\Sigma_k, \gamma_k)}$ such that
\[ \frac{1}{8\pi} \int_{\Sigma_j} H_{g,j} d\sigma_j = \LL_{(\Sigma_j, \gamma_j)} \]
for some $j = 1, \ldots, k$, then $k = 1$ and $(\Omega, g)$ is isometric to a mean-convex handlebody with flat interior whose genus is that of $\Sigma_1$ (since $k = 1$). In particular, if $\operatorname{genus}(\Sigma_1) = 0$ then $(\Omega, g)$ is an Alexandrov embedded mean-convex ball in $\R^3$.
\end{thm}

\begin{rem}
The equality above will hold true for all $j = 1, \ldots, k$ if
\[ \frac{1}{8\pi} \int_{\p \Omega} H_g d\sigma = \LL_{(\Sigma_1,\gamma_1),\ldots,(\Sigma_k,\gamma_k)} \text{,} \]   
i.e., if the fill-in $(\Omega,g)$ attains the supremum $\LL_{(\Sigma_1,\gamma_1),\ldots,(\Sigma_k,\gamma_k)}$. 
\end{rem}

Theorem \ref{thm-main-rigidity} confirms that a disconnected boundary cannot support a maximizing configuration, as is the case in Theorem \ref{thm-ShiTam} when the boundary consists of spheres with positive Gauss curvature.

\begin{rem}
It would be interesting to know whether Theorem \ref{thm-main-finiteness} continues to hold if we replace the boundary 2-sphere with 
a surface of higher genus. Note that Theorems \ref{thm-main-additivity} and \ref{thm-main-rigidity} do not make any genus assumptions.
\end{rem}

\subsection{Quasi-local mass}

The results above, together with the implication of Theorem \ref{thm-ShiTam} on Brown-York mass, suggest a variational analog of this quasi-local mass that does not require positivity of the Gauss curvature on the boundary. We describe this in the following Theorem-Definition.

\begin{thm}[Definition of $\m(\Sigma;\Omega)$] \label{df-BY-analog}
Given a compact, connected Riemannian 3-manifold $(\Omega, g)$ 
with nonnegative scalar curvature, and a mean-convex  boundary $ \Sigma$ which is a topological 2-sphere, define
\be
\m  (\Sigma; \Omega) = \L_{ (\Sigma, \gamma) } - \frac{1}{8\pi} \int_{\Sigma} H_g d \sigma .
\ee
Here $\gamma$, $d\sigma$ are the metric and the area element induced on $\Sigma$ by $g$,
and,  
\be \label{eq-def-L}
\L_{(\Sigma, \gamma)} := \sup \lf\{ \frac{1}{8 \pi}  \int_{\p M} H_h d \sigma \ | \ (M, h) \in \F \ri\} ,
\ee
where $\F = \F_{(\Sigma,\gamma)}$ is as in Definition \ref{df-F} below. Then $\m(\Sigma;\Omega)$ is
\begin{enumerate}
\item[(a)] well-defined and finite,
\item[(b)] nonnegative, i.e.,  $\m (\Sigma; \Omega) \ge 0$, and
\item[(c)] $\m  (\Sigma; \Omega) = 0$ only if $ (\Omega, g)$ is a flat 3-ball immersed in $\R^3$.
\end{enumerate}
Finally, when $\Sigma$ has positive Gauss curvature, $\m  (\Sigma; \Omega) = \mby (\Sigma; \Omega)$.
Hence, $\m  (\Sigma; \Omega)$ may be viewed as a variational analog 
and generalization of Brown-York mass. 
\end{thm}

\begin{df}[Fill-ins, II]
\label{df-F}
Let $\Sigma$ be a closed, connected, orientable surface endowed with a Riemannian metric $\gamma$.
Denote by $\F_{(\Sigma, \gamma)}$ the set of all compact, connected
 Riemannian 3-manifolds $(\Omega, g)$ with boundary such that:
\begin{enumerate}
\item $\partial \Omega$ has a connected  component $\So$ which, with the induced metric, is isometric to $(\Sigma, \gamma)$,
\item $\So$ is mean-convex; i.e., the mean curvature vector of $\So$ points inward, 
\item $\p \Omega \setminus \So$, if nonempty, is a minimal surface, possibly disconnected, and
\item $R(g) \geq 0$, where $R(g)$ is the scalar curvature of $g$.
\end{enumerate}
\end{df}

\begin{rem} 
Clearly $\FF_{(\Sigma,\gamma)} \subseteq \F_{(\Sigma,\gamma)}$, and therefore $\LL_{(\Sigma,\gamma)} \leq \L_{(\Sigma,\gamma)}$. We will, in fact, prove that these last two quantities coincide. Thus, $\m(\Sigma; \Omega)$ will remain unchanged if one replaces $\F$ with $\FF$. We nevertheless choose to use this  enlarged class in the definition of $\m(\Sigma; \Omega)$, because it is more suitable for discussion on quasi-local mass and more convenient for the cut-and-fill operations:

\begin{enumerate}
\item 
For an  element $(\Omega, g) \in \F_{(\Sigma,\gamma)}$, the portion $\p \Omega \setminus \So$ of the boundary, if nonempty, 
represents  horizons of black holes detected by observers at (the outer boundary) $\So$.
For example, the region bounded by a rotationally symmetric sphere $(\Sigma, \gamma)$ of positive mean curvature and the horizon in a half spatial Schwarzschild manifold of positive mass is now a valid fill-in of $(\Sigma, \gamma)$, while it wasn't in the original class $\FF_{(\Sigma,\gamma)}$.
\item 
Elements in $\F_{(\Sigma,\gamma)}$ serve as building blocks, effectively allowing to deduce  $\FF_{(\Sigma_1,\gamma_1),\ldots,(\Sigma_k,\gamma_k)}$-related  results from them via a cut-and-fill technique; we cut composite manifolds across minimal surfaces, and, when necessary, fill in holes with 3-balls of positive scalar curvature.
\end{enumerate}
\end{rem}

\begin{rem}
The motivation behind Theorem \ref{df-BY-analog} is that only the integral quantity 
$ \int_{\partial \Omega} H_0 d\sigma$
is of actual interest for the purposes of the actual definition of Brown-York mass in \eqref{eq-BY}, 
and not the pointwise quantity $H_0$. Moreover, when $\Sigma$ has positive Gauss curvature, Theorem \ref{thm-ShiTam} characterizes this integral quantity as 
$$ \int_{\Sigma} H_0 d\sigma = \sup_{(M, h) \in \FF} \int_{\partial M} H_h d\sigma ,$$ 
over an appropriate class $\FF$ of fill-ins.
\end{rem}

\begin{rem}
It seems a challenging question to check whether mean-convex domains $\Omega \subset \R^3$, with $\Sigma = \p \Omega$ a 2-sphere, necessarily maximize the total mean curvature on their boundary relative to all competitors in $\FF_{(\Sigma,g_{\R^3}|_{\Sigma})}$; i.e., is $\m (\Sigma; \Omega) = 0$ for all mean-convex domains $\Omega \subset \R^3$ which are bounded by a 2-sphere $\Sigma$? If $ \Sigma$ is strictly convex in $\R^3$, 
then, indeed, $\m(\Sigma; \Omega) = \mby(\Sigma; \Omega) = 0$ by Theorem \ref{thm-ShiTam}.
\end{rem}

\subsection*{Organization of the paper}

In Section \ref{sec-technical-lemmas} we establish the basic cutting, filling, and doubling lemmas that will be used throughout the paper. In Section \ref{sec-single-component} we establish, for $\F$ and $\L$, finiteness and rigidity results in the spirit of Theorems \ref{thm-main-finiteness} (finiteness) and \ref{thm-main-rigidity} (rigidity). In Section \ref{sec-multiple-components} we first prove Theorem \ref{thm-main-additivity} (additivity) for $\FF$, $\LL$, and use it to derive Theorems \ref{thm-main-finiteness} and \ref{thm-main-rigidity} for $\FF$ and $\LL$ from the corresponding  results for $\F$ and $\L$. In Section \ref{sec-BY-mass} we prove
Theorem \ref{df-BY-analog} pertaining to $\m(\Sigma; \Omega)$.

\begin{rem}
After the first version of this paper was completed, 
we learned that the assertion  $\LL_{(\Sigma, \gamma)} < \infty$  in Theorem \ref{thm-main-finiteness}, and its proof, appear independently in a recent work by 
Lu \cite{Lu-isometric-embedding}, which establishes a priori estimates for certain isometric embeddings of 2-spheres into general Riemannian 3-manifolds.
\end{rem}

\subsection*{Acknowledgments} The first named author would like to thank Rick Schoen for his continued guidance and Otis Chodosh for a helpful discussion in the early stages of this work. Both authors would like to thank Pengfei Guan  for bringing the work in \cite{Lu-isometric-embedding} to our attention, as well as the referees for their helpful comments and their careful reading of the manuscript.

\section{Technical lemmas} \label{sec-technical-lemmas}

In this section we prove some technical lemmas that will be invoked multiple times throughout the paper.

\begin{lemma}[Cutting] \label{lem-cutting}
    Let $\Sigma_1, \ldots, \Sigma_k$ be $k \geq 2$ closed, connected, orientable surfaces endowed with Riemannian metrics $\gamma_1, \ldots, \gamma_k$. If $(\Omega, g) \in \FF_{(\Sigma_1,\gamma_1),\ldots,(\Sigma_k,\gamma_k)}$, then for every $j = 1,\ldots,k$ there exists an
    $(\Omega_j, g_j) \in \F_{(\Sigma_j,\gamma_j)}$ such that
    \begin{enumerate}
        \item the mean curvature of $\Sigma_j$ is the same in $(\Omega, g)$ and $(\Omega_j, g_j)$,
        \item $R(g_j) > 0$ on $\Omega_j$ if $R(g) > 0$ on $\Omega$, and
        \item $\p \Omega_j \setminus \Sigma_j$ consists of  stable, orientable minimal surfaces.
    \end{enumerate}
\end{lemma}
\begin{proof}
For convenience, first suppose $ \Omega$ is orientable.
Write $S_j$ for the boundary component of $\Omega$ corresponding to our fixed $\Sigma_j$. Since the mean curvature vector points inward on $\p \Omega$, standard results in geometric measure theory show that there exists a smooth, oriented minimal surface $S$ in the interior of $\Omega$ that minimizes area in the homology class $[S_j] \in H_2(\Omega; \Z)$. (Specifically, we minimize area in the class of integral currents homologous to $S_j$.) Denote $\Omega_j$ the metric completion of the component of $\Omega \setminus S$ containing $S_j$, and $g_j$ its induced metric. Note that $\Omega_j$ is orientable. Then $(\Omega_j, g_j)  \in \F_{(\Sigma_j,\gamma_j)}$ and satisfies all three assertions.

If $ \Omega$ is nonorientable, consider its orientation double cover $\pi:  \tilde \Omega \rightarrow \Omega$. 
Using the fact that $ S_j$ itself  is orientable, one can show that  $ \pi^{-1} (S_j)$ is  a  disjoint union of  two  copies  of $ S_j$. 
Denote one of them by $ \tilde{S}_{j}$. 
The lemma follows by repeating  the previous proof with $(\Omega, g)$ and $S_j$ replaced by $(\tilde \Omega, \pi^*(g))$ and  $\tilde{S}_{j}$, respectively.
\end{proof}

\begin{lemma}[Filling] \label{lem-filling}
    Let $\Sigma$ be a closed, connected, orientable surface  with a Riemannian metric $\gamma$. Suppose $(\Omega, g) \in \F_{(\Sigma,\gamma)}$ is such that:
    \begin{enumerate}
        \item $\Sigma_H = \p \Omega \setminus \So$ is nonempty,
        \item $R(g) > 0$ on $\Sigma_H \subset \p \Omega$, and
        \item every component of $\Sigma_H$ is a stable minimal 2-sphere.
    \end{enumerate}
    Then for every $\eta > 0$, there exists $(D, h) \in \FF_{(\Sigma,\gamma)}$ with $H_h > H_g - \eta$ on its boundary $S_O$, which corresponds to $\So$.
\end{lemma}
\begin{proof}
The first eigenvalue of the operator 
$$ -\Delta_{T_\ell} + K_\ell $$
is strictly positive on each sphere $T_\ell$ on the portion $\Sigma_H$ of $\p \Omega$ (cf. \cite{F-S, M-S}). 
Here $\Delta_{T_\ell}$, $K_\ell$ are the Laplacian, and Gauss curvature of $T_\ell$ with respect to the induced metric from $g$. 
By the method in \cite{M-S},  we can glue 3-balls of positive scalar curvature onto  $T_\ell$ to obtain a compact manifold $D$ with $\partial D = S_O$.
More precisely, for every $\ell$ one can apply  Lemma 1.3  in \cite{M-S} to first produce a cylinder of positive scalar curvature, 
and then attach a spherical cap to it. 
If the resulting metric $\hat{g}$ on $D$ were smooth across every $T_\ell$, then $(D, \hat{g}) \in \FF_{(\Sigma, \gamma)}$. In general $\hat{g}$ will not be smooth across the $T_\ell$, so we apply  \cite[Proposition 3.1]{Miao02} to $(D, \hat{g})$ followed by a small conformal deformation to obtain
another metric $h$ on $D$ such that $(D, h) \in \FF_{(\Sigma, \gamma)}$ and $H_{S_O,h} \geq H_{S_O,\hat{g}} - \eta/2 \geq H_{\Sigma_O,g} - \eta$.
\end{proof}

\begin{lemma}[Doubling] \label{lem-doubling}
Let $\Sigma$ be a closed, connected, orientable surface with a Riemannian metric $\gamma$. 
Suppose $(\Omega, g) \in \F_{(\Sigma,\gamma)}$, and that $\Sigma_H = \p \Omega \setminus \So$ is nonempty.  Let $D$ denote the doubling of $\Omega$ across $\Sigma_H$, so that $\p D = \So \cup \So'$, where $\So'$ denotes the mirror image of $\So$. For every $\eta > 0$ there exists a scalar-flat Riemannian metric $h$ on $D$ such that $(D, h)$ satisfies:
\begin{enumerate} 
\item $\So$ with the induced metric from $h$ is isometric to $(\Sigma,\gamma)$,
\item $H_h > H_g$ on $\So$, and
\item $H_h' > H_g - \eta$ on $\So'$,
\end{enumerate}
where $H_g$ denotes the mean curvature of $\So$ in $(\Omega, g)$, and $H_h$, $H_h'$ denote the mean curvatures of $\So$, $\So'$ in $(D, h)$.
\end{lemma}
\begin{proof}
    The following argument is motivated by  that used by  Jauregui in the proof of \cite[Proposition 7] {JJ}. 

    First, we reduce to the case of $R(g)$ being identically zero. If $R(g) \ge  0 $ but $ R(g)$ is not identically $0$,
consider a conformally deformed metric $\tilde g = u^4 g $ on $\Omega$ where $ u > 0 $ is the unique solution  to  
\be 
\lf\{ 
\begin{array}{rcl}
\Delta_g u - \frac18 R(g) u & = & 0  \ \  \mathrm{on} \ \Omega \\
u & = & 1  \ \ \mathrm{at} \ \So \\
\frac{\partial u}{\partial \nu} & = & 0 \ \ \mathrm{at} \ \Sh .
\end{array}
\ri.
\ee
Here $ \Delta_g $ is the Laplacian of $g$, and $\nu$ is the outward unit normal to $\p \Omega$ with respect to $g$.
Then $H_{\tilde{g}} = 0$ at $\Sh$, and $H_{\tilde{g}} = H_g + 4 \frac{\p u}{\p \nu}$ 
at $ \So$.
It follows from  the strong maximum principle and the fact $\frac{\partial u}{\partial \nu} = 0 $ at $ \Sh$
that the maximum of $ u$ is attained at $ \So$.
 Hence, by the the strong maximum principle again, $  \frac{\p u}{\p \nu} > 0 $ and 
$H_{\tilde g} > H_g $ at $ \So$.
 At this point, we replace our $(\Omega, g)$ with $(\Omega, \tilde{g})$. If $R(g)$ were identically zero to begin with, take $\tilde{g} = g$.

In any case, $H_{\tilde{g}} \geq H_g$ on $\So$ and $H_{\tilde{g}} = 0$ on $\Sh$. Given a small constant $ \varepsilon \in (0,1)$, 
 let $ \phi_1  > 0 $  
 be the harmonic function  on $ \Omega$ such that $ \phi_1 = 1$ at $\So$ and $ \phi_1 = 1 - \frac12 \varepsilon $ at $ \Sh$.
Let $  \phi_2 = ( 2 - \varepsilon) - \phi_1$.
Consider two conformally deformed metrics 
$ g_1 = \phi_1^4 \tilde{g} $ and $ g_2 = \phi_2^4 \tilde{g} $.
They satisfy the following properties:
\begin{enumerate}
\item[(i)] the induced metrics on $\Sh$ from  $ g_1$ and $g_2$ agree,
\item[(ii)] the mean curvature 
of $ \Sh$ in $(\Omega, g_1)$ with respect to the inward normal agrees with the mean curvature of $\Sh$ 
in $(\Omega, g_2)$ with respect to the outward normal,
\item[(iii)] the mean curvature $H_{g_1} $ of $\So$ in $(\Omega, g_1)$ 
satisfies  $H_{g_1}  > H_{\tilde{g}}$ by the strong maximum principle,  and
\item[(iv)] the mean curvature $H_{g_2}$ of $ \So$ in $(\Omega, g_2)$ remains positive and arbitrarily close to $H_{\tilde{g}}$, say $H_{g_2} > H_{\tilde{g}} - \eta/2$, if $\varepsilon $ is small enough. 
\end{enumerate}
Now attach $(\Omega, g_1)$ and $(\Omega, g_2)$ along $\Sh$, and call  the resulting manifold $(D, h)$.
In $(D, h)$, denote $ \So $ coming from $(\Omega, g_1)$ still by $ \So$ while denote $\So $ coming from
$ (\Omega, g_2)$ by $ \So'$.
If the metric $h$ were smooth across $\Sh$, then it satisfies all the properties  required. 
In general, we can replace $h$  with another metric that is obtained by 
applying \cite[Proposition 3.1]{Miao02} to $(D, h)$ at $ \Sh$ 
followed by a small conformal deformation that fixes the boundary.  The result follows.
\end{proof}

\section{$\F$ and $\L$ for a prescribed surface}

\label{sec-single-component}

We first prove

\begin{prop}[Finiteness for $\F$] \label{thm-main-finiteness-single}
Let $\Sigma$ be a 2-sphere.
Given a Riemannian metric $\gamma$ on $\Sigma$, 
there exists a constant $C > 0 $, depending only on $\gamma$, such that 
\be \label{eq-main-upper-bd}
\sup_{(\Omega, g) \in \F_{(\Sigma, \gamma)}} \int_{\p \Omega} H_g d \sigma < C.
\ee
As a result, 
\[ \L_{(\Sigma, \gamma)} := \sup \lf\{ \frac{1}{8 \pi}  \int_{\p \Omega} H_g d \sigma \ | \ (\Omega, g) \in \F \ri\} < \infty . \]
\end{prop}

To prove Proposition \ref{thm-main-finiteness-single}, we make use of 
the following result of Shi and Tam in \cite{ShiTam06},  which  is built on the work on Wang and Yau in \cite{WangYau06}.

\begin{thm} [\cite{ShiTam06}] \label{thm-ShiTam-H}
Let $(\Omega, g)$ be a compact,  orientable 
Riemannian 3-manifold  with boundary $\Sigma$, and scalar curvature $R(g) \ge - 6 \kappa^2$ for some constant $ \kappa > 0$.
Suppose $\Sigma $   is a topological 2-sphere with  Gauss curvature $ K > - \kappa^2$ and positive mean curvature $H$.
Let $\iota: \Sigma \rightarrow   \H^3_{- \kappa^2} $ be an isometric embedding and denote $ \Sigma_0 = \iota (\Sigma)$, which is a convex 
surface in $  \H^3_{- \kappa^2}$. 
Let $ D \subset  \H^3_{- \kappa^2} $ be the bounded region  enclosed by $ \Sigma_0$.
Then for any $ p \in  D$, 
\be  \label{eq-t-ShiTam-H}
  \int_{\Sigma} H  \cosh \kappa r(p,  \iota  (z) ) d \sigma (z) 
\le 
\int_{\Sigma_0} H_0   \cosh \kappa r (p, y)  d \sigma_0 (y) ,
\ee
where $ r ( p , \cdot) $ denotes the distance to $ p $ in $  \H^3_{- \kappa^2}$,   
$ d \sigma_0 $ is the area element on $ \Sigma_0$, and  $H_{0}$ is the mean curvature  of $\Sigma_0$ in $ \H^3_{- \kappa^2}$.
Moreover, equality in \eqref{eq-t-ShiTam-H} holds if and only if  $(\Omega, g)$ is isometric to the  domain  bounded by 
$\Sigma_0$ in $\H^3_{- \kappa^2}$.
\end{thm}

\begin{rem}
The existence of such an embedding $\iota$ is given by a  theorem of Pogorelov \cite{Pogorelov}.
\end{rem}

\begin{rem}
Even though not explicitly stated, the proof of Theorem \ref{thm-ShiTam-H} in \cite{ShiTam06} assumes $\Omega$ is orientable since
the proof uses the fact that $\Omega$ is a spin manifold. 
\end{rem}

As the first step toward proving Proposition \ref{thm-main-finiteness-single}, we want to point out that Theorem \ref{thm-ShiTam-H} continues to hold 
 for manifolds with more than one boundary component under suitable 
boundary conditions; moreover, it holds without the orientability assumption. 

\begin{prop} \label{prop-ShiTam-H-1}
Let $(\Omega, g)$ be a compact, orientable
Riemannian 3-manifold  with boundary $\p \Omega$, and scalar curvature $R(g) \ge - 6 \kappa^2$ for some constant $ \kappa > 0$.
Suppose 
\begin{enumerate}
\item[(a)]  $ \p \Omega$ has a connected component $\Sigma $ that 
is a topological 2-sphere with  Gauss curvature $ K > - \kappa^2$ and  mean curvature $H > 0$, and
\item[(b)] $ \p \Omega \setminus \Sigma$, if nonempty, has mean curvature $ H \ge - 2 \kappa$.
\end{enumerate} 
Then the conclusion of Theorem \ref{thm-ShiTam-H} holds for such an $(\Omega, g)$. 
\end{prop}

\begin{proof} 
As $ \Omega$ is an oriented 3-manifold, $\Omega$ is spin. 
In \cite[Theorem 4.7]{C-H}, Chru\'{s}ciel and Herzlich proved that 
if $(M, g)$ is an $n$-dimensional, spin, asymptotically hyperbolic manifold with a compact boundary $\p M$ 
such that  
$$ R(g) \ge - n(n-1) \kappa^2  \ \ \mathrm{and} \ \  H \ge - (n-1) \kappa , $$
where  $R(g)$ is the scalar curvature of $g$ and $H$ is the mean curvature of $\p M$ (with respect to the outward normal), 
then  the positive mass theorem holds on such an $(M, g)$. 
Now going through the same proof of Theorem \ref{thm-ShiTam-H} in \cite{ShiTam06}  which involves 
carrying out the same construction in \cite{WangYau06}, but replacing the positivity of the  mass expression used in \cite{WangYau06}
 by the positivity of the mass provided in  \cite[Theorem 4.7]{C-H}, one concludes that 
Theorem \ref{thm-ShiTam-H} holds for such an $(\Omega, g)$. 
\end{proof}

\begin{prop} \label{prop-ShiTam-H}
Proposition \ref{prop-ShiTam-H-1} continues to hold without the orientability assumption on $ \Omega$.
\end{prop}

\begin{proof}
Let $ (\Omega, g)$ be a compact Riemannian 3-manifold satisfying all assumptions in Proposition \ref{prop-ShiTam-H-1}
 except that $\Omega$ is nonorientable.
 Let $ \tilde \Omega $ be the orientation double cover of $ \Omega$ and let 
 $\pi : \tilde \Omega \rightarrow \Omega$ 
 be the corresponding covering map. 
 It is easily seen that $ \p \tilde \Omega = \pi^{-1} ( \p \Omega)$ and 
 $ \p \tilde \Omega  $ doubly covers  $ \p \Omega$. 
 Let $ S  = \p \Omega \setminus \Sigma$, which can be empty.
 Let $ \tilde{S}  = \pi^{-1} (S ) $ and 
 $ \tilde \Sigma = \pi^{-1} (\Sigma)$. Since $ \Sigma$ is a 2-sphere and 
 $ \tilde \Sigma $ doubly covers $ \Sigma$, $\tilde \Sigma$ is the disjoint union of 
 two 2-spheres, which we denote by $\tilde \Sigma^{(1)}$ and $ \tilde \Sigma^{(2)}$. 
 Now let $ \tilde g = \pi^* g $ on $ \tilde \Omega$, which has scalar curvature $ R (\tilde g) \ge - 6 \kappa^2$.
Let $ \tilde H $ denote the mean curvature of $ \p \tilde \Omega = \tilde S \cup  \tilde \Sigma^{(1)} \cup \tilde \Sigma^{(2) }  $
in $(\tilde \Omega, \tilde g)$  (with respect to the outward normal).
Then 
$ \tilde H \ge - 2 \kappa $ on $ \tilde S$ and $ \tilde H > 0 $ on $ \Sigma^{(i)}$, $ i =1, 2$. 
Hence, Proposition \ref{prop-ShiTam-H-1} is applicable to  $(\tilde \Omega, \tilde g)$. 
Therefore,   the conclusion of Theorem \ref{thm-ShiTam-H} holds for each $ \tilde \Sigma^{(i)}$, $i=1,2$, which
in turn shows that the same is true for $\Sigma$ in $(\Omega, g)$. 
\end{proof}

\begin{proof}[Proof of Proposition \ref{thm-main-finiteness-single}]
Let $ \kappa > 0 $ be a constant such that
$ K > - \kappa^2  $
where $ K$ is the Gauss curvature of $\gamma$.
Let $ \iota$ be an isometric embedding of $(\Sigma, \gamma)$ in $ \H^3_{-\kappa^2} $.
Let $ D  \subset  \H^3_{-\kappa^2}  $  be the bounded region  enclosed by   
the convex surface $ \Sigma_0  =  \iota (\Sigma)$.  
Let $ H_0 > 0 $ be the mean curvature of $ \Sigma_0 $ in $  \H^3_{-\kappa^2}$ 
and let $ p \in D$ be a fixed  point.

Take $(\Omega, g) \in \F_{(\Sigma, \gamma)}$, then
$ R(g) \ge 0 > - 6 \kappa^2 $ in $\Omega$ 
and
$ H = 0 > - 2 \kappa $ at $ \p \Omega \setminus \So$.
By Proposition \ref{prop-ShiTam-H}, the conclusion of Theorem \ref{thm-ShiTam-H} holds
for  $(\Omega, g)$ at $\So$, i.e.  
\be  \label{eq-ShiTam-H}
 \int_{\So} H  \cosh \kappa r(p,  \iota \circ \phi  (z) ) d \sigma (z) 
\le 
\int_{\Sigma_0} H_0   \cosh \kappa r (p, y)  d \sigma_0 (y) ,
\ee
where   $ \phi$ is a  given isometry between $\So$ and $ (\Sigma, \gamma)$. 
Let 
$$ r_{*} = \min \{  r ( p , y) \ | \ y \in \Sigma_0 \} > 0 ,$$
 then \eqref{eq-ShiTam-H} gives 
\be \label{eq-upper-bd-1}
\int_{\So} H d \sigma \le  ( \cosh \kappa r_{*})^{-1}  \int_{\Sigma_0} H_0   \cosh \kappa r (p, y)  d \sigma_0 (y).
\ee
Clearly, the right-hand side of \eqref{eq-upper-bd-1} is a  constant determined only by $ (\Sigma, \gamma)$. 
This proves Proposition \ref{thm-main-finiteness-single}.
\end{proof}

\begin{rem}[Weak mean-convexity]
In  Proposition \ref{thm-main-finiteness-single}, the assumption $ H_g > 0$  at $\So$ for an $(\Omega, g) \in \F_{(\Sigma,\gamma)}$
can be relaxed to $ H_g \ge 0$. 
To see this, suppose $ (\Omega, g)$ 
satisfies all the assumptions
imposed on an ``fill-in" in Definition \ref{df-F}, 
except that $\So$ is  weakly mean-convex, i.e., $ H_g \ge 0$ at $\So$.
 Let $ \phi \ge 0 $ be a fixed, nontrivial  function on $ \Omega$. Let $ w $ be the unique solution to 
\be
\lf\{ 
\begin{array}{rcl}
\Delta_g w - \frac18 R(g) w & = \ \phi  &   \mathrm{on} \ \Omega \\
w & = \  0  &   \mathrm{at} \ \p \Omega .
\end{array}
\ri.
\ee
Then $ w \le 0 $ on $ \Omega$ and  $ \frac{\p w}{\p \nu} >  0 $ at $\p \Omega$
by the strong maximum principle, where $ \nu $ is the outward unit normal 
to $ \p \Omega$ in $(\Omega, g)$. 
Given a small $\varepsilon > 0 $,  consider  $ g_\varepsilon  = ( 1 + \varepsilon w )^4 g $.
Then
$ H_{g_\varepsilon} = H_g + 4 \varepsilon \frac{\p w}{\p \nu} > 0 $ at $ \So $, 
$ H_{g_\varepsilon} = 4 \varepsilon  \frac{\p w}{\p \nu} > 0 $ at $ \p \Omega \setminus \So$, 
and
\be \label{eq-Rge}
R (g_\varepsilon) = - ( 1 + \varepsilon w )^{-5} \lf[ \varepsilon ( 8  \Delta_g w - R(g) w) - R(g)   \ri].
\ee
Now we repeat the proof of Proposition \ref{thm-main-finiteness-single},  with $g$ replaced by $ g_\varepsilon$.
Note that the assumptions $R(g) \ge 0 $
and $H=0 $ at $\p \Omega \setminus \So$  in the  
proof of Proposition \ref{thm-main-finiteness-single} are only used to 
yield  $ R(g) > - 6 \kappa^2$ and $ H> - 2 \kappa$ at $\p \Omega \setminus \So$.
On the other hand, if $\varepsilon $ is sufficiently small, 
 $ R(g_\varepsilon) > - 6 \kappa^2 $ by 
\eqref{eq-Rge} and the fact  $R(g) \ge 0$. 
We already know $H_{g_\varepsilon} > 0$ at $\p \Omega \setminus \So$.
Therefore, the same proof leading to \eqref{eq-upper-bd-1} gives 
\be \label{eq-He}
 \int_{\So} H_{g_\varepsilon}  \, d \sigma < C 
 \ee
where $ C$ is the quantity  that is on the right-side of \eqref{eq-upper-bd-1}.
As  $ H_g < H_{g_\varepsilon}$, we conclude \eqref{eq-He} holds with $H_{g_\varepsilon}$ replaced by $H_g$.
Thus, the claim in this remark follows.
\end{rem}

Next, we want to prove

\begin{prop} [Rigidity in  $\F$] \label{thm-main-rigidity-single}
Let $\Sigma$ be an arbitrary closed, connected, orientable surface, endowed with a Riemannian metric $\gamma$. 
If there exists $(\Omega, g) \in \F_{(\Sigma, \gamma)}$ attaining the
supremum $\L_{(\Sigma,\gamma)}$, i.e., if
\[  \frac{1}{8\pi} \int_{\p \Omega} H_g d \sigma = \L_{(\Sigma, \gamma)} \text{,} \]
then $(\Omega, g)$ is isometric to a mean-convex handlebody with flat interior, whose genus is that of $\Sigma$. In particular, if $\operatorname{genus}(\Sigma_1) = 0$ then $(\Omega, g)$ is an Alexandrov embedded mean-convex ball in $\R^3$.
\end{prop}

One  step in proving Proposition \ref{thm-main-rigidity-single} is to 
exclude the possibility of having minimal boundary components appear on a maximizing fill-in. We do this by using minimal boundary components as a tool that helps us increase the mean curvature of $\So$.

\begin{lemma} \label{lem-no-horizon}
Let $(\Omega, g) \in \F_{(\Sigma, \gamma)}$. If $ \p \Omega \setminus \So \neq \emptyset$, then 
 $$ \frac{1}{8 \pi} \int_{\p \Omega} H_g d \sigma <  \L_{(\Sigma, \gamma)} .$$ 
\end{lemma} 
\begin{proof}
Employ Lemma \ref{lem-doubling} to double $(\Omega, g)$ across $\Sigma_H = \p \Omega \setminus \So$
if it is not empty.  For small  $\eta > 0$, the resulting manifold $(D, h)$ is going to have the mean curvature $H_h$ of $\So$ in $(D, h) $ satisfying $H_h > H_g$
and the mean curvature $ H_h' $ of $ \So'$ in $(D, h) $ satisfying $H_h' > 0$; here $\So'$ is the mirror image of $\So$ in $D$. The metric on $\So$ is still $\gamma$, while the metric on $\So'$ may have changed to some $\gamma'$. Cut $(D, h)$ via Lemma \ref{lem-cutting} to isolate the boundary component $\So$ and obtain $(M, h) \in \F_{(\Sigma,\gamma)}$. We compare the original fill-in $(\Omega, g)$ with the new fill-in $(M, h)$, which is by construction in $\F_{(\Sigma,\gamma)}$, and for which the new mean curvature $H_h$ on $\So$ exceeds the original mean curvature $H_g$ pointwise; as a result,
\[ \frac{1}{8 \pi} \int_{\p \Omega} H_g d \sigma <  \frac{1}{8 \pi} \int_{\p M} H_h d\sigma \leq \L_{(\Sigma, \gamma)} \text{,} \]
and the claim follows.
\end{proof}

\begin{proof}[Proof of Proposition \ref{thm-main-rigidity-single}]
Suppose 
$$ \int_{\p \Omega}  H_g d \sigma   =  8 \pi \L_{(\Sigma, \gamma)}$$
for some $ (\Omega, g) \in \F_{(\Sigma,\gamma)}$.
By Lemma \ref{lem-no-horizon},  $ \p \Omega = \So$. 
We claim  $ R(g) = 0 $, which again can be seen by  applying a  conformal deformation. 
Let $ u $ be the  (unique) positive solution to 
\be
\lf\{ 
\begin{array}{rcl}
\Delta_g u - \frac18 R(g) u & = & 0  \ \  \mathrm{on} \ \Omega \\
u & = & 1  \ \ \mathrm{at} \ \p \Omega  .
\end{array}
\ri.
\ee
Consider the conformally deformed metric
$  \hat g = u^4 g $.
Then $ R(\hat g) = 0 $, $ \hat g  = g $ at $ \p \Omega$, and 
$H_{\hat{g}} = H_g + 4 \frac{\p u}{\p \nu} $, 
where $ \nu $ is  the unit outward normal to $ \p \Omega$ in $(\Omega, g)$. If $ u $ is not identically a constant, then 
$ \frac{\p u}{\p \nu} > 0 $ at  $ \p \Omega$ 
by the strong maximum principle. Hence, $(\Omega , \hat g) \in \F_{(\Sigma,\gamma)}$ and
\[ \L_{(\Sigma, \gamma)} = \frac{1}{8\pi} \int_{\p \Omega}  H_g d \sigma < \frac{1}{8\pi} \int_{\p \Omega} H_{\hat{g}}  d \sigma \leq \L_{(\Sigma, \gamma)} \text{,} \]
a contradiction.
Therefore, $ u $ is constant, which shows $ R(g) = 0 $ on $ \Omega$. 

Now consider the space of metrics  $\bar g$ on $ \Omega$ given by 
\be
\M_g^0 = \{  \bar g \ | \ R( \bar g) = 0 \ \mathrm{and} \ \bar g |_{T \p \Omega} = g |_{T \p \Omega} \} ,
\ee
where $ \bar g |_{T \p \Omega }$ denotes  the induced metric on $ \p \Omega  $ from $  \bar g $. 
 Since  $ g \in \M_g^0 $ maximizes  the total boundary mean curvature in $\F$,
one knows that $ g$ is a critical point of the functional 
$$ \bar g \mapsto \int_{\p \Omega} H_{\bar{g}} d \sigma, \ \ \bar g \in  \M^0_g .$$
Therefore, by  \cite[Corollary 2.1]{MiaoShiTam09} (also cf. \cite[Lemma  4]{EMW}), 
$g$ is Ricci flat, and hence is flat  as $\Omega$ is 3-dimensional.

It remains for us to check the topological conclusion. First we claim that there are no closed embedded minimal surfaces (oriented or not) in the interior of $\Omega$. We proceed by contradiction.

If there were such a minimal surface $T$ then by compactness there would exist an interior smooth geodesic $\Gamma : [0,\ell] \to \Omega$ joining a pair of closest points between $T$ and $\partial \Omega$; $\Gamma(0) \in T$, $\Gamma'(0) \perp T$, $\Gamma(\ell) \in \partial \Omega$, and $\Gamma'(\ell) \perp \partial \Omega$. The second variation of the length of this geodesic summed among a basis of two unit normal variations $V_i$, $i = 1, 2$, is
\begin{align*} 
 \sum_{i=1}^2 \delta^2 \Gamma(V, V) & = -\int_0^{\ell} \Ric_g(\Gamma'(s), \Gamma'(s)) ds - H_{\partial \Omega,g}(\Gamma(\ell)) - H_{T,g}(\Gamma(0)) \\
    & = -H_{\partial \Omega,g}(\Gamma(\ell)) < 0 \text{,}
\end{align*}
since $g$ is flat and $T$ is minimal. This means $\Gamma$ is unstable, in contradiction with its minimizing nature. The claim follows.

Theorem 1 and Proposition 1  in \cite{MeeksSimonYau} tell us that in the absence of interior minimal surfaces, $\Omega$ is necessarily a  handlebody
with mean-convex boundary. 

Finally, when $\operatorname{genus}(\Sigma) = 0$ then we know $\Omega$ a genus-$0$ handlebody, i.e., a 3-ball. We've shown its metric $g$ is flat, so we can locally (and therefore globally since it is simply connected) immerse $\Omega$ in $\R^3$.
\end{proof}

\section{$\FF$ and $\LL$ for multiple prescribed surfaces}

\label{sec-multiple-components}

In this section we prove the theorems pertinent to the $\LL$ functional 
on compact, mean-convex, 3-manifolds with nonnegative scalar curvature and fixed boundary geometry consisting of possibly multiple components.

\begin{proof}[Proof of Theorem \ref{thm-main-additivity} (additivity)]
Let $(\Omega, g) \in \FF_{(\Sigma_1, \gamma_1), \ldots, (\Sigma_k, \gamma_k)}$, with $k \geq 2$. For each $j = 1, \ldots, k$,  denote the boundary component in $\partial \Omega$ corresponding to $\Sigma_j$ by $S_j$.

Let $ \phi <  0$ be a fixed function on $\Omega$ and let $w$ be the unique solution to 
\be
\lf\{ 
\begin{array}{rcll}
\Delta_{g} w - \frac18 R(g) w & = & \phi  &    \mathrm{on} \ \Omega \\
w & = & 0   &  \mathrm{at} \ \p \Omega .
\end{array}
\ri.
\ee
For small $\tau > 0$, consider the metrics 
 ${g}^{(\tau)} = ( 1 + \tau w)^4 g$.
 Then $g^{(\tau)} $ has  strictly positive scalar curvature on $ \Omega$,
 $\p \Omega$ has positive mean curvature $H_{g^{(\tau)}}$ in  $(\Omega, {g}^{(\tau)})$, and 
$$ \int_{S_j} H_{g^{(\tau)},j} d {\sigma}_j \geq \int_{S_j} H_{g,j} d\sigma_j - \varepsilon $$
if $\tau $ is small enough. 
Here $ \epsilon > 0$ is  any  given constant.
Write $\tilde{g}$ for this $g^{(\tau)}$.

Employing Lemma \ref{lem-cutting}, cut $(\Omega, \tilde{g})$ to isolate the boundary component $S_j$ and obtain 
$(\Omega_j, \tilde{g}_j) \in \F_{(\Sigma_j,\gamma_j)}$ with $R(\tilde{g}_j) > 0$, $H_{\tilde{g}_j} = H_{\tilde{g},j}$ on $S_j$ and $\p \Omega_j = S_j \cup T_j$, with $T_j$ a nonempty union of smooth, stable, oriented minimal surfaces. Moreover it follows from \cite{F-S} and $R(\tilde{g}_j) > 0$ that $T_j$ consists of 2-spheres. Next, employ Lemma \ref{lem-filling} (filling) once for every sphere in $T_j$ to replace $(\Omega_j, \tilde{g}_j)$ with $(M_j, h_j) \in \FF_{(\Sigma_j,\gamma_j)}$, with the mean curvature $H_{h_j}$ of $S_j$ in $(M_j, h_j)$ satisfying
\[ \int_{S_j} H_{h_j} d\sigma_j \geq \int_{S_j} H_{\tilde{g}_j} d\sigma_j - \varepsilon \geq \int_{S_j} H_{g,j} d\sigma_j - 2\varepsilon \text{.} \]
Since $(M_j, h_j) \in \FF_{(\Sigma_j, \gamma_j)}$, the left hand side is bounded from above by $\LL_{(\Sigma_j, \gamma_j)}$. Rearranging, we have
\[ \int_{S_j} H_{g,j} d\sigma_j \leq 2 \varepsilon + 8 \pi \LL_{(\Sigma_j, \gamma_j)} \text{.} \]
Letting $\varepsilon \downarrow 0$, we obtain \eqref{ineq-single-sup}.

Notice that if we had not let $\varepsilon \downarrow 0$, and instead carried out the procedure above for all $j = 1, \ldots, k$, and summed over $j$, then
\[ \int_{\p \Omega} H_g d\sigma \leq 2 k \varepsilon + 8 \pi \sum_{j=1}^k \LL_{(\Sigma_j,\gamma_j)} \text{.} \]
Letting $\varepsilon \downarrow 0$, and recalling that $(\Omega, g) \in \FF_{(\Sigma_1, \gamma_1), \ldots, (\Sigma_k, \gamma_k)}$ was arbitrary, we conclude
\[ \LL_{(\Sigma_1,\gamma_1),\ldots,(\Sigma_k,\gamma_k)} \leq \sum_{j=1}^k \LL_{(\Sigma_j,\gamma_j)} \text{.} \]

It remains to check the reverse direction ``$\geq$'' in \eqref{eq-additivity}. We will assume that all quantities on the right are finite, as
a similar argument carries through to the general case.
 Let $\varepsilon > 0$ be given. For each $j = 1, \ldots, k$, let $(\Omega_j, g_j) \in \FF_{(\Sigma_j, \gamma_j)}$ be such that
\[ \int_{\partial \Omega_j} H_{g_j} d\sigma_j \geq \LL_{(\Sigma_j, \gamma_j)} - \varepsilon \text{.} \]
On each $ \Omega_j$, let $ \phi_j <  0$ be a fixed function and let $w_j$ be the unique solution to 
\be
\lf\{ 
\begin{array}{rcll}
\Delta_{g_j} w_j - \frac18 R(g_j) w_j & = & \phi_j  &    \mathrm{on} \ \Omega_j \\
w_j & = & 0   &  \mathrm{at} \ \p \Omega_j  .
\end{array}
\ri.
\ee
For small $\tau > 0$, consider the metrics 
 ${g}^{(\tau)}_j = ( 1 + \tau w_j)^4 g_j$.
 Then $g^{(\tau)}_j $ has  strictly positive scalar curvature on $ \Omega_j$,
 $\p \Omega_j$ has positive mean curvature $H_{g_j^{(\tau)}}$ in  $(\Omega_j, {g}^{(\tau)}_j)$, and 
$$ \int_{\partial \Omega_j} H_{g_j^{(\tau)}} d {\sigma}_j \geq \LL_{(\Sigma_j, \gamma_j)} - 2\varepsilon $$
if $\tau $ is small enough.
Applying the connect-sum construction for positive scalar curvature manifolds (cf. \cite{SchoenYau79} and \cite{G-L}),
we obtain $\Omega = \Omega_1 \# \ldots \# \Omega_k$  endowed  with a metric ${g}$ of positive scalar curvature that coincides with 
${g}^{(\tau)}_j$ near each  $ \p \Omega_j$. In particular, $(\Omega, g) \in \FF_{(\Sigma, \gamma_1),\ldots, (\Sigma_k , \gamma_k) }$
and  it satisfies 
\[ \int_{\partial \Omega} {H}_g d {\sigma} \geq \sum_{j=1}^k \LL_{(\Sigma_j,\gamma_j)} - 2k\varepsilon \text{.} \]
Letting $\varepsilon \downarrow 0$, we conclude that  ``$\geq$'' holds.
\end{proof}

Using the tools developed so far, we can easily complete the proofs of Theorems \ref{thm-main-finiteness} (finiteness) and \ref{thm-main-rigidity} (rigidity).

\begin{proof}[Proof of Theorem \ref{thm-main-finiteness}]
Since $\FF_{(\Sigma,\gamma)} \subset \F_{(\Sigma,\gamma)}$, we see that
\be \label{eq-single-1}
 \LL_{(\Sigma,\gamma)} \leq \L_{(\Sigma,\gamma)} < \infty 
 \ee
for all Riemannian metrics $\gamma$ on a 2-sphere $ \Sigma$, 
where  the rightmost inequality follows from Proposition \ref{thm-main-finiteness-single}.
In Section \ref{sec-BY-mass} we will show that the leftmost inequality is, in fact, an equality.

When $k \geq 2$, and $\gamma_1, \ldots, \gamma_k$ are all metrics on a 2-sphere $\Sigma$, then
\[ \LL_{(\Sigma, \gamma_1), \ldots, (\Sigma, \gamma_k)} = \sum_{j=1}^k \LL_{(\Sigma,\gamma_j)} < \infty \]
by Theorem \ref{thm-main-additivity} and \eqref{eq-single-1}.
\end{proof}

\begin{proof}[Proof of Theorem \ref{thm-main-rigidity}]
If $k = 1$,  the proof of Proposition \ref{thm-main-rigidity-single} carries through verbatim in this case (except we don't need to invoke Lemma \ref{lem-no-horizon}, 
since we have no minimal boundary components).

If $k \geq 2$, then our assumption is that
\[ \frac{1}{8\pi} \int_{S_j} H_{g,j} d\sigma_j = \LL_{(\Sigma_j,\gamma_j)} \]
holds for some fixed $j \in \{1, \ldots, k\}$, where $S_j$ represents the boundary component corresponding to the $\Sigma_j$ on which we have equality.
Employ Lemma \ref{lem-cutting} to isolate the boundary component $S_j$ and obtain $(\Omega_j, g_j) \in \F_{(\Sigma_j, \gamma_j)}$, with $\p \Omega \setminus S_j \neq \emptyset$. Then employ Lemma \ref{lem-doubling} to double $\Omega_j$ across $\p \Omega_j \setminus S_j$ and obtain $(D_j, h_j)$. Writing $S_j'$ for the mirror image of $S_j$ under the doubling, we have $\p D_j = S_j \cup S_j'$. By construction, 
$(D_j,h_j) \in \FF_{(\Sigma_j,\gamma_j),(\Sigma_j,\gamma_j')}$ for some metric $\gamma_j'$, so by \eqref{ineq-single-sup} in Theorem \ref{thm-main-additivity},
\[ \frac{1}{8\pi} \int_{S_j} H_{h_j} d\sigma_j \leq \LL_{(\Sigma_j,\gamma_j)} \text{.} \]
Also by construction, the mean curvature $H_{h_j}$ of $S_j$ in $(D_j, h_j)$ exceeds the original mean curvature $H_{g,j}$, so,
\[ \LL_{(\Sigma_j,\gamma_j)} = \frac{1}{8\pi} \int_{S_j} H_{g,j} d\sigma_j < \frac{1}{8\pi} \int_{S_j} H_{h_j} d\sigma_j \leq \LL_{(\Sigma_j,\gamma_j)} \text{,} \]
a contradiction.
\end{proof}

\section{Application to $\m(\Sigma;\Omega)$}
\label{sec-BY-mass}

We recall and prove Theorem \ref{df-BY-analog}.

\begin{thm*}
Given a compact Riemannian 3-manifold $(\Omega, g)$ 
with nonnegative scalar curvature, and a mean-convex  boundary $ \Sigma$ which is a topological 2-sphere, define
\[ \m  (\Sigma; \Omega) = \L_{ (\Sigma, \gamma) } - \frac{1}{8\pi} \int_{\Sigma} H_g d \sigma . \]
Then $\m(\Sigma;\Omega)$ is
\begin{enumerate}
\item[(a)] well-defined and finite,
\item[(b)] nonnegative, i.e.,  $\m (\Sigma; \Omega) \ge 0$, and
\item[(c)] $\m  (\Sigma; \Omega) = 0$ only if $ (\Omega, g)$ is a flat 3-ball immersed in $\R^3$.
\end{enumerate}
Finally, when $\Sigma$ has positive Gauss curvature, $\m  (\Sigma; \Omega) = \mby (\Sigma; \Omega)$.
\end{thm*}
\begin{proof}
    Finiteness follows from Proposition \ref{thm-main-finiteness-single}, and nonnegativity is true by the definition of $\L_{(\Sigma, \gamma)}$. The rigidity statement follows from Proposition \ref{thm-main-rigidity-single}. Finally, when $\Sigma$ has positive Gauss curvature, Theorem \ref{thm-ShiTam} implies 
$$ \int_{\Sigma} H_0 d\sigma = \sup_{(M, h) \in \FF_{(\Sigma, \gamma)}} \int_{\partial M} H_h d\sigma = 8 \pi \LL_{(\Sigma, \gamma)} $$ 
and the result follows by Proposition \ref{prop-invariance} below.
\end{proof}

\begin{prop}
\label{prop-invariance}
    Let $\Sigma$ be a closed, connected, orientable surface with a Riemannian metric $\gamma$. Then
    \[ \LL_{(\Sigma,\gamma)} = \L_{(\Sigma,\gamma)} \text{.} \]
    In other words, $\F_{(\Sigma,\gamma)}$ and $\L_{(\Sigma,\gamma)}$ can be replaced by $\FF_{(\Sigma,\gamma)}$ and $\LL_{(\Sigma,\gamma)}$ in the definition of $\m(\Sigma;\Omega)$.
\end{prop}
\begin{proof}
It suffices to show $ \L_{(\Sigma,\gamma)} \le \LL_{(\Sigma,\gamma)} $.
Pick $(\Omega, g) \in \F_{(\Sigma, \gamma)}$ 
such that $\Sigma_H = \p \Omega \setminus \So$ is nonempty.
Let $ h$ be a metric given by Lemma \ref{lem-doubling} on the doubling $ D$ of $ \Omega$ across $\Sh$.
Then $ H_h > H_g $ on $ \So $. 
Applying Theorem \ref{thm-main-additivity} to $(D, h)$ at $\So$, we have 
$$ \int_{\So} H_h d \sigma \le 8 \pi \LL_{(\Sigma, \gamma)} .$$
Therefore, 
$$ \int_{\p \Omega} H_g d \sigma = \int_{\So}  H_g d \sigma <  \int_{\So} H_h d \sigma \le 8 \pi \LL_{(\Sigma, \gamma)}, $$ 
which proves the claim.
\end{proof}

\end{document}